\documentclass[11pt]{amsart}
\usepackage[a4paper,textwidth=155mm,textheight=248mm]{geometry}
\usepackage[T1]{fontenc}
\usepackage{lmodern,microtype,amsmath,amssymb,mathtools}
\usepackage[hidelinks]{hyperref}
\hypersetup{
  pdftitle={Half-sized gaps between zeros of the Riemann zeta function},
  pdfauthor={Andriy Bondarenko and Winston Heap},
  pdfsubject={Small gaps between zeros of the Riemann zeta function},
  pdfkeywords={Riemann hypothesis, zeta zeros, small gaps, resonance}
}

\newtheorem{theorem}{Theorem}
\newtheorem{lemma}{Lemma}
\DeclareMathOperator{\Ree}{Re}
\DeclareMathOperator{\Imm}{Im}
\newcommand{\R}{\mathbb R}

\title[Half-sized gaps between zeta zeros]{Half-sized gaps between zeros of the Riemann zeta function}
\author{Andriy Bondarenko}
\author{Winston Heap}
\address{Department of Mathematical Sciences, Norwegian University of Science and
Technology (NTNU), 7491 Trondheim, Norway}
\email{andriybond@gmail.com}
\email{winstonheap@gmail.com}
\date{}

\begin{document}
\begin{abstract}
Assuming the Riemann hypothesis, we prove that the lower limit of the
normalised gaps between consecutive zeros of the Riemann zeta function
is at most $1/2$. Our methods utilise a more efficient method for detecting gaps with short Dirichlet polynomials. Specifically, we apply resonance to a quotient of shifted zeta
functions, rather than to the number of zeros in a short interval,
using a M\"obius resonator adapted to the symmetry between the shifts.
\end{abstract}
\maketitle

\section{Introduction}

Let $0<\gamma_1\leq\gamma_2\leq\cdots$ denote the positive ordinates
of the nontrivial zeros of the Riemann zeta function, counted with
multiplicity. The average spacing of consecutive zeros at height $T$ is asymptotic to $2\pi/\log T$ and thus it is natural to ask how much one can deviate from this. Regarding smaller gaps, one defines
\[
 \mu=\liminf_{n\to\infty}
       \frac{(\gamma_{n+1}-\gamma_n)\log\gamma_n}{2\pi}.
\]
It is conjectured that $\mu=0$, that is, consecutive gaps can be
arbitrarily small in comparison with their average.

Assuming the Riemann hypothesis (RH), Conrey, Ghosh and Gonek \cite{CGG} showed that $\mu<0.5172$. Later, improvements of Bui, Milinovich and
Ng~\cite{BMN} gave $\mu<0.5155$; then to
$\mu<0.5154$ by Feng and Wu~\cite{FW}, and to $\mu<0.515396$ by
Preobrazhenski\u{\i}~\cite{P}. More recently, Inoue~\cite{Inoue}
obtained $\mu<0.50895$. These bounds were obtained by different
variations of what can be thought of as the resonance method, beginning with the
method of Montgomery and Odlyzko~\cite{MO}. Here, one utilises a Dirichlet polynomial to pick out non-zero values of the counting function of the number of zeros in a short interval.  
Prior to Inoue, much work sought to optimise the Dirichlet polynomial coefficients in this method;
Inoue's improvement altered the initial setup, combining resonance techniques with a quadratic expression in the
number of zeros in a short interval, overcoming previous inefficiencies. For related results on positive
proportions of small spacings and on gaps between distinct zeros,
see Bui, Goldston, Milinovich and Montgomery~\cite{BGMM}.

It is now well known that there is a limitation to the usual Montgomery--Odlyzko criterion.
Goldston, Trudgian and Turnage-Butterbaugh~\cite{GTT} showed that,
when the Dirichlet polynomial has length $T^{1-\varepsilon}$ with
fixed $\varepsilon>0$, no choice of coefficients allows this
criterion to reach the half-spacing bound, with further refinements of Inoue, Kobayashi and Toma \cite{IKT} pushing the limitations of the method to $\mu<0.508$. Here we prove the following.

\begin{theorem}\label{thm:main}
Assume the Riemann hypothesis. Then $\mu\leq1/2$.
\end{theorem}

Our approach is somewhat different. In the classical method one
seeks to make a weighted average of $N_c(u)$, the number of zeros in
an interval of normalised length $c$ centred at $u$, exceed one. We instead work with
the quotient
\begin{equation}\label{eq:F}
 F(s)=-i\,
 \frac{e^{i\pi(c-1/2)}\zeta(s+i\pi c/\log T)
                       -\zeta(s-i\pi c/\log T)}
      {e^{i\pi(c-1/2)}\zeta(s+i\pi c/\log T)
                       +\zeta(s-i\pi c/\log T)},
\end{equation}
where, under the contrary assumption $\mu>1/2$, we choose
$1/2<c<\min(1,\mu)$. The gap assumption gives $N_c(u)\leq1$,
which, as we shall see, implies the bound $|F|\leq1+o(1)$ in a suitable rectangle.
Utilising resonance allows us to, essentially, exhibit larger values than this, giving the required contradiction. It is the use of this quotient, rather than a refinement of the usual
weighted average of $N_c(u)$, that allows our improvements. 

We give the notation and state the two lemmas needed for the proof
in Section~2. The theorem follows immediately from them. Their proofs
are given in Section~3.

\medskip
\noindent\textbf{Acknowledgments.} A. Bondarenko was supported in
part by Grant 334466 of the Research Council of Norway. The authors
gratefully acknowledge the assistance of OpenAI's ChatGPT, whose
exploratory discussions, calculations and numerics helped stimulate
the development of this paper.

\section{Preliminaries and the main idea}

Throughout the proof we assume RH and, seeking a contradiction,
suppose that $\mu>1/2$. Fix
\begin{equation}\label{eq:c}
 \frac12<c<\min(1,\mu),
\end{equation}
and let $T$ tend to infinity.

For $v\in\R$, let $N_c(v)$ count the zero ordinates in
\[
 \left[v-\frac{\pi c}{\log T},\,
       v+\frac{\pi c}{\log T}\right],
\]
including negative ordinates and counting multiplicities. For
sufficiently large $T$, the condition $c<\mu$ gives
\begin{equation}\label{eq:binary}
 0\leq N_c(v)\leq1\qquad(T/2\leq v\leq3T).
\end{equation}
Indeed, two zeros in such an interval would give a consecutive gap
of normalised size at most $c+O(1/\log T)$, contradicting $c<\mu$.

Now, the quotient $F$ in~\eqref{eq:F} can be written as
\begin{equation}\label{eq:tangent}
 F(s)=\tan\left(\frac{\pi}{2}\left(z(s)-\frac12\right)\right),
\end{equation}
where, to the right of the critical line at the heights under
consideration,
\begin{equation}\label{eq:z}
 z(s)=c-\frac{i}{\pi}
       \bigl(\log\zeta(s+i\pi c/\log T)
                    -\log\zeta(s-i\pi c/\log T)\bigr).
\end{equation}
The logarithms are the standard continuations from the Euler product, and RH allows us to claim their holomorphy up to the line $\Re(s)=1/2$.
As we shall see, $\Ree z(s)$ is, up to $O(1/\log T)$, a Poisson
average of $N_c(v)$. This follows by the usual arguments involving the Hadamard product and Stirling's formula
applied to the completed zeta function
\[
 \xi(s)=\tfrac12s(s-1)\pi^{-s/2}\Gamma(s/2)\zeta(s).
\]
From this we will see that~\eqref{eq:binary} places $\Ree z(s)$ between zero and one,
apart from a small error. Then since the tangent in~\eqref{eq:tangent} maps
the strip $0<\Ree z<1$ onto the unit disc, we roughly acquire $|F(s)|\leq 1$. We sum these properties up in the following Lemma.

\begin{lemma}\label{lem:bound}
Assume RH and fix $1/2<c<\min(1,\mu)$. For sufficiently large $T$,
the function $F$ is holomorphic in
\[
 \frac12<\Ree s\leq2,\qquad
 \frac{3T}{4}\leq\Imm s\leq\frac{9T}{4},
\]
extends continuously to the left edge, and satisfies
\begin{equation}\label{eq:Fbound}
 |F(s)|\leq1+\frac{6\pi}{\log T}
\end{equation}
uniformly on the closed rectangle.
\end{lemma}

We now consider $F$ in the mean, using the resonance method to pick out large values. This will eventually lead to a contradiction with the pointwise bounds above. 
As is typical, we use a short M\"obius resonator. Let $\mu(n)$ denote the
M\"obius function. For every $n\in\mathbb N$, define
\begin{align}
 r(n)&=\mu(n)\cos\left(\frac{\pi}{2}\left(c-\frac12\right)
                              +\frac{\pi c\log n}{\log T}\right),
       \label{eq:r}\\
 q(n)&=\mu(n)\sin\left(\frac{\pi}{2}\left(c-\frac12\right)
                              +\frac{\pi c\log n}{\log T}\right).
       \label{eq:q}
\end{align}
We consider the corresponding Dirichlet polynomials
\begin{equation}\label{eq:polynomials}
 R(s)=\sum_{n\leq T^{1/(2c)}}\frac{r(n)}{n^s},
 \qquad
 Q(s)=\sum_{n\leq T^{1/(2c)}}\frac{q(n)}{n^s}.
\end{equation}
To see why these coefficients are natural, observe that, for
$\Ree s>1$,
\begin{align}
 \sum_{n\geq1}\frac{r(n)}{n^s}
   &=\frac12\left(
       \frac{e^{i\pi(c-1/2)/2}}{\zeta(s-i\pi c/\log T)}
       +\frac{e^{-i\pi(c-1/2)/2}}{\zeta(s+i\pi c/\log T)}
                \right),\label{eq:reciprocalR}\\
 \sum_{n\geq1}\frac{q(n)}{n^s}
   &=\frac1{2i}\left(
       \frac{e^{i\pi(c-1/2)/2}}{\zeta(s-i\pi c/\log T)}
       -\frac{e^{-i\pi(c-1/2)/2}}{\zeta(s+i\pi c/\log T)}
                \right).\label{eq:reciprocalQ}
\end{align}
The function $F$ has an absolutely convergent Dirichlet series
\[
 F(s)=\sum_{n\geq1}\frac{a(n)}{n^s},\qquad \Ree s\geq2.
\]
By~\eqref{eq:F}, \eqref{eq:reciprocalR} and~\eqref{eq:reciprocalQ},
\begin{equation}\label{eq:identity}
 F(s)\sum_{n\geq1}\frac{r(n)}{n^s}
       =\sum_{n\geq1}\frac{q(n)}{n^s}
       \qquad(\Ree s\geq2).
\end{equation}
Therefore the coefficient of $n^{-s}$ in $F(s)R(s)$
equals $q(n)$ for every $n\leq T^{1/(2c)}$.

In our integrals, we use the Gaussian weight
\begin{equation}\label{eq:weight}
 W(t)=\frac{\log T}{\sqrt{2\pi}\,T}
       \exp\left(-\frac{(t-3T/2)^2(\log T)^2}{2T^2}\right).
\end{equation}
Note that $ W(t)\geq 0$, $\widehat{W}(0)= \int_\R W(t)\,dt=1$ and 
\[
 \int_{\R\setminus[3T/4,9T/4]}W(t)\,dt
       \ll e^{-(\log T)^2/4}.
\]

Now, for comparison, the Montgomery--Odlyzko criterion requires the
following quotient to exceed one; see~\cite[\S2]{BMN}
and~\cite[Theorem~3]{Inoue}. For our choice of $R$, however,
\[
 \begin{aligned}
 \frac{\displaystyle\int_\R N_c(t)|R(1/2+it)|^2W(t)\,dt}
      {\displaystyle\int_\R|R(1/2+it)|^2W(t)\,dt}
 &=c+\frac1\pi\int_0^\pi\frac{\sin v}{v}\,dv
       -\frac{2}{\pi\bigl(\pi+2\cos(\pi c)\bigr)}+o(1)\\
 &\longrightarrow
   \frac12+\frac1\pi\int_0^\pi\frac{\sin v}{v}\,dv
       -\frac{2}{\pi^2}
   =0.8868475\ldots<1,
 \end{aligned}
\]
as $T\to\infty$ and then $c\downarrow1/2$.
We instead integrate $FR\overline Q$ against $W$.
The following lemma is very natural in view of~\eqref{eq:identity}.

\begin{lemma}\label{lem:mean}
Assume RH and fix $1/2<c<\min(1,\mu)$. As $T\to\infty$,
\begin{equation}\label{eq:mixed}
 \int_{3T/4}^{9T/4}F(1/2+it)R(1/2+it)
       \overline{Q(1/2+it)}W(t)\,dt
   =\sum_{n\leq T^{1/(2c)}}\frac{q(n)^2}{n}+o(1).
\end{equation}
Moreover,
\begin{align}
 \int_\R|R(1/2+it)|^2W(t)\,dt
   &=\sum_{n\leq T^{1/(2c)}}\frac{r(n)^2}{n}+o(1),
       \label{eq:Rnorm}\\
 \int_\R|Q(1/2+it)|^2W(t)\,dt
   &=\sum_{n\leq T^{1/(2c)}}\frac{q(n)^2}{n}+o(1),
       \label{eq:Qnorm}
\end{align}
Finally,
\begin{equation}\label{eq:coeffnorms}
 \sum_{n\leq T^{1/(2c)}}\frac{q(n)^2-r(n)^2}{n}
   =-\frac{6\cos(\pi c)}{\pi^3c}\log T+O(1).
\end{equation}
\end{lemma}

Note that both inequalities $1/2<c<\mu$ are used in Lemma~\ref{lem:mean}.
The condition $c<\mu$ excludes poles of $F$ during the contour
shift, by Lemma~\ref{lem:bound}, whilst the condition $c>1/2$ gives
$T^{1/(2c)}=o(T/\log T)$, making the off-diagonal terms negligible.
We remark that neither the coefficient identity~\eqref{eq:identity} nor the
evaluations~\eqref{eq:Rnorm}--\eqref{eq:coeffnorms} require the
gap assumption. With this Lemma in hand, we can now obtain our desired contradiction.

\begin{proof}[Proof of Theorem~\ref{thm:main}]
Suppose that $\mu>1/2$ and choose $c$ as in~\eqref{eq:c}.
The two mean squares in Lemma~\ref{lem:mean} are $O(\log T)$,
since $|r(n)|,|q(n)|\leq1$. Using Lemma~\ref{lem:bound} and
Cauchy--Schwarz, we obtain
\begin{equation}\label{eq:normcomparison}
 \begin{aligned}
 \sum_{n\leq T^{1/(2c)}}\frac{q(n)^2}{n}
 &=\left|\int_{3T/4}^{9T/4}F(1/2+it)R(1/2+it)
              \overline{Q(1/2+it)}W(t)\,dt\right|+o(1)\\
 &\leq\left(1+\frac{6\pi}{\log T}\right)
   \left(\int_\R|R(1/2+it)|^2W(t)\,dt\right)^{1/2}\\
 &\qquad\times
   \left(\int_\R|Q(1/2+it)|^2W(t)\,dt\right)^{1/2}+o(1)\\
 &\leq\frac12\sum_{n\leq T^{1/(2c)}}
               \frac{r(n)^2+q(n)^2}{n}+O(1).
 \end{aligned}
\end{equation}
Thus the difference of the coefficient sums is bounded above.
But~\eqref{eq:coeffnorms} gives
\[
 \sum_{n\leq T^{1/(2c)}}\frac{q(n)^2-r(n)^2}{n}
   =-\frac{6\cos(\pi c)}{\pi^3c}\log T+O(1)
   \longrightarrow+\infty,
\]
since $1/2<c<1$. This is a contradiction.
\end{proof}

\section{Proof of lemmas}

In this section we will prove Lemma~\ref{lem:bound} and
Lemma~\ref{lem:mean}.

\begin{proof}[Proof of Lemma~\ref{lem:bound}]
Recall that
\[
 \begin{aligned}
 F(s)&=-i\,
 \frac{e^{i\pi(c-1/2)}\zeta(s+i\pi c/\log T)
                       -\zeta(s-i\pi c/\log T)}
      {e^{i\pi(c-1/2)}\zeta(s+i\pi c/\log T)
                       +\zeta(s-i\pi c/\log T)}\\
 &=\tan\left(\frac\pi2\left(z(s)-\frac12\right)\right),
 \end{aligned}
\]
where
\[
 z(s)=c+\frac1{\pi i}
       \int_{s-i\pi c/\log T}^{s+i\pi c/\log T}
             \frac{\zeta'}{\zeta}(w)\,dw.
\]
We shall bound $F$ by estimating $\Ree z(s)$.
Write $s=1/2+u+it$, with $0<u\leq3/2$ and
$3T/4\leq t\leq9T/4$. Under RH, the Hadamard product and
Stirling's formula give
\begin{equation}\label{eq:logderivative}
 \Ree\frac{\zeta'}{\zeta}(1/2+u+iv)
   =\sum_\gamma\frac{u}{u^2+(v-\gamma)^2}
       -\frac12\log\frac{v}{2\pi}+O(1/T),
\end{equation}
uniformly for $T/2\leq v\leq3T$; see~\cite[\S2]{Sound}.
Integrating between the two shifts and using the definition of
$N_c(v)$, we obtain
\begin{equation}\label{eq:completion}
 \Ree z(s)
   =\int_\R\frac{u\,N_c(v)}{\pi\bigl(u^2+(t-v)^2\bigr)}\,dv
     +\frac{c}{\log T}\log\frac{2\pi T}{t}
     +O\left(\frac{1}{T\log T}\right).
\end{equation}
The interchange with the sum over zeros is justified by
nonnegativity, and the implied constants are absolute for
$1/2<c<1$.

Here we use $c<\mu$: by~\eqref{eq:binary}, $0\leq N_c(v)\leq1$
on $[T/2,3T]$. Outside this interval the Riemann--von Mangoldt
formula gives $N_c(v)\ll\log(|v|+3)$, and hence
\[
 \int_{\R\setminus[T/2,3T]}
       \frac{u\,N_c(v)}{\pi\bigl(u^2+(t-v)^2\bigr)}\,dv
 \ll\frac{\log T}{T}.
\]
The Poisson kernel is nonnegative and has integral one, so the
integral in~\eqref{eq:completion} lies between zero and
$1+O(\log T/T)$. Moreover,
\[
 0<c\log\frac{2\pi T}{t}
   \leq c\log\frac{8\pi}{3}<3c.
\]
It follows that, for sufficiently large $T$,
\begin{equation}\label{eq:strip}
 -\frac{3c}{\log T}\leq\Ree z(s)\leq1+\frac{3c}{\log T}.
\end{equation}
Thus the strip estimate holds with the explicit constant $C=3c$.
The error terms above are uniform for $1/2<c<1$. Since $c<\mu$,
\eqref{eq:binary} holds for all sufficiently large $T$.

The strip in~\eqref{eq:strip} excludes the poles
$z=3/2+2k$, $k\in\mathbb Z$, of the tangent in~\eqref{eq:tangent}.
Thus $F$ is holomorphic, and for sufficiently large $T$,
\[
 |F(s)|
   \leq\tan\left(\frac\pi4+\frac{3\pi c}{2\log T}\right)
   \leq1+\frac{6\pi}{\log T}.
\]
Finally, the quotient defining $F$ is meromorphic
across the critical line, and the bound from the right excludes
a pole on that line. Any apparent singularities there are
therefore removable, and the same estimate holds on the closed
rectangle.
\end{proof}

\begin{proof}[Proof of Lemma~\ref{lem:mean}]
On $\Ree s=2$, write $F(s)R(s)=\sum_{m\geq1}f(m)m^{-s}$,
where $\sum_m|f(m)|m^{-2}\ll1$ and, by~\eqref{eq:identity},
$f(n)=q(n)$ for $n\leq T^{1/(2c)}$. Put
\[
 H(s)=F(s)R(s)Q(1-s)
             W\left(\frac{s-1/2}{i}\right).
\]
Lemma~\ref{lem:bound} permits us to move the contour to
$\Ree s=2$, because $F$, and hence $H$, is holomorphic in the rectangle
\[
 \left\{s\in\mathbb C:
   \frac12<\Ree s\leq2,\quad
   \frac{3T}{4}\leq\Imm s\leq\frac{9T}{4}\right\}
\]
and continuous up to its left edge. The horizontal integrals and
the tails are negligible by Gaussian decay, so
\begin{equation}\label{eq:contourshift}
 \frac1i\int_{1/2+3iT/4}^{1/2+9iT/4}H(s)\,ds
   =\frac1i\int_{2-i\infty}^{2+i\infty}H(s)\,ds+o(1).
\end{equation}
For $x>0$, the Gaussian integral gives
\begin{equation}\label{eq:kernel}
 \frac1i\int_{2-i\infty}^{2+i\infty}
       x^{-s}W\left(\frac{s-1/2}{i}\right)\,ds
   =x^{-1/2-3iT/2}
       \exp\left(-\frac{T^2\log^2x}{2(\log T)^2}\right).
\end{equation}
Hence
\begin{equation}\label{eq:diagonalsum}
 \begin{aligned}
 \frac1i\int_{2-i\infty}^{2+i\infty}H(s)\,ds
 &=\sum_{n\leq T^{1/(2c)}}\sum_{m\geq1}
       \frac{f(m)q(n)}{\sqrt{mn}}
       \left(\frac mn\right)^{-3iT/2}
       \exp\left(-\frac{T^2\log^2(m/n)}{2(\log T)^2}\right)\\
 &=\sum_{n\leq T^{1/(2c)}}\frac{q(n)^2}{n}+o(1).
 \end{aligned}
\end{equation}
Here the off-diagonal terms are $o(1)$ by Gaussian decay, since
$|\log(m/n)|\geq(2T^{1/(2c)})^{-1}$ for $m\ne n$ and $c>1/2$.
Now~\eqref{eq:mixed} follows from~\eqref{eq:contourshift}
and~\eqref{eq:diagonalsum}, and the same calculation
gives~\eqref{eq:Rnorm} and~\eqref{eq:Qnorm}.

Finally, partial summation with
$\sum_{n\leq x}\mu(n)^2/n=(6/\pi^2)\log x+O(1)$ gives
\[
 \begin{aligned}
 \sum_{n\leq T^{1/(2c)}}\frac{q(n)^2-r(n)^2}{n}
 &=-\frac6{\pi^2}\log T\int_0^{1/(2c)}
              \sin(\pi c+2\pi cu)\,du+O(1)\\
 &=-\frac{6\cos(\pi c)}{\pi^3c}\log T+O(1),
 \end{aligned}
\]
which proves~\eqref{eq:coeffnorms}.
\end{proof}

\end{document}